\documentclass[11pt,a4paper]{article}
\usepackage[T1]{fontenc}
\usepackage[round,authoryear]{natbib}
\usepackage{lmodern,amsmath,amssymb,amsthm,booktabs,longtable,geometry,microtype,placeins}
\usepackage[colorlinks=true,linkcolor=blue,citecolor=blue,urlcolor=blue]{hyperref}
\hypersetup{pdftitle={The complete existence spectrum of periodic complementary ternary pairs of weight nine},pdfauthor={Stelios D. Georgiou and Stella Stylianou}}
\newtheorem{theorem}{Theorem}[section]
\newtheorem{lemma}[theorem]{Lemma}
\newtheorem{proposition}[theorem]{Proposition}
\newtheorem{corollary}[theorem]{Corollary}
\theoremstyle{definition}\newtheorem{definition}[theorem]{Definition}
\theoremstyle{remark}
\newcommand{\Z}{\mathbb Z}\newcommand{\Q}{\mathbb Q}
\newcommand{\supp}{\operatorname{supp}}
\newcommand{\wt}{\operatorname{wt}}
\title{The complete existence spectrum of periodic complementary ternary pairs of weight nine}
\author{Stelios D. Georgiou\textsuperscript{*} and Stella Stylianou\\
\small School of Science, RMIT University, Melbourne, Australia\\
\small \textsuperscript{*}Corresponding author: \href{mailto:stelios.georgiou@rmit.edu.au}{stelios.georgiou@rmit.edu.au}\\
\small ORCID: \href{https://orcid.org/0000-0001-5096-0067}{0000-0001-5096-0067} (S.D.G.);\quad
\href{https://orcid.org/0000-0003-3677-2830}{0000-0003-3677-2830} (S.S.)}
\date{}
\begin{document}\maketitle
\begin{abstract}
For two sequences with entries $0,\pm1$, periodic complementarity means that their periodic autocorrelation functions sum to zero at every nonzero shift. We give necessary and sufficient conditions on the length for the existence of such pairs of total weight nine, where weight counts nonzero entries. The conditions decide existence at any prescribed length by divisibility tests. The main new result bounds the length of primitive pairs with two nonzero components. A pair is primitive if it cannot be shortened by removing uniformly inserted zeros after independent cyclic permutations. We prove primitive-length bounds of $64,48,64$ for weight distributions $(5,4),(6,3),(7,2)$. These bounds make the classification finite and prove that the earlier searches contain every required case. Exhaustive enumeration gives $91$ primitive sequence-equivalence classes. Including the five known circulant classes gives $96$ representatives from which every pair can be generated. The same catalogue gives class counts at every length. For constructing one pair, $33$ displayed seeds suffice. The results settle earlier conjectures on the weight distributions and decide exactly which orders admit the associated two-circulant weighing-matrix construction.
\end{abstract}

\section{Introduction}
A ternary sequence has entries $0,1,-1$, and its weight is the number of nonzero entries. Following \citet{GK2002,G2010}, a pair of sequences of the same length has \emph{zero PAF} if the sum of their periodic autocorrelation functions (PAFs) is zero at every nonzero cyclic shift. Such a pair is also called \emph{periodically complementary}. This condition concerns the sum; the individual PAFs need not vanish. Their weight distribution $(r,s)$ records the weights of the two components. We determine all lengths admitting such pairs of total weight $r+s=9$, and describe every pair under sequence equivalence.

The problem arises in the construction of weighing matrices. A weighing matrix $W(v,k)$ is a square matrix of order $v$ with entries $0,\pm1$ and satisfies $WW^{\mathsf T}=kI_v$, where $I_v$ is the identity matrix of order $v$. A complementary pair of length $n$ and total weight nine supplies the first rows of two circulant matrices that give a $W(2n,9)$~\citep{GS1979,GK2002}. Thus an existence criterion for the pairs decides exactly when this construction can be used. A complete catalogue supplies all its sequence inputs, which can then be used in further weighing-matrix constructions. We do not classify matrices under arbitrary signed row and column permutations.

\citet{GK2002} investigated weight nine by separating its possible distributions. \citet{G2010} used signed differences to classify the distributions $(5,4)$ and $(6,3)$ through length $400$, and $(7,2)$ through length $250$. \citet{GDK2012} extended the $(5,4)$ calculation to lengths $400$--$500$. These searches supplied constructions and finite nonexistence results. Uniform zero insertion also gave infinite families from the known examples. The remaining question was whether a larger length could contain a pair outside those families. The scalar case $(9,0)$, in which the second component is zero, was already settled by the theory of circulant weighing matrices~\citep{AAMS2008,SS2013}.

Uniform zero insertion with an integer factor $t\ge1$ places $t-1$ zeros between consecutive entries of each component, including between the last and first entries. It changes length $m$ to $tm$ and preserves zero PAF and the component weights. We also call this operation \emph{inflation}; its converse is \emph{contraction}. To distinguish new starting pairs from repetitions of this construction, we use the following definition.

\begin{definition}\label{def:equivalence}
A cyclic permutation adds a fixed integer to all positions modulo the sequence length.
Pairs of the same component weights are \emph{equivalent} under independent cyclic permutations, reversals and sign changes of the components, together with a common position multiplier $j\mapsto uj\pmod n$, where $\gcd(u,n)=1$. We call this \emph{sequence equivalence}, following \citet{G2010}, and order the components by decreasing weight. A pair is \emph{primitive} if independent cyclic permutations cannot put it in the form obtained from a shorter pair by uniform zero insertion with factor $t>1$.
\end{definition}

Primitive pairs are needed for classification, not just for finding one example at each length. A composite length can admit a primitive pair even when a shorter pair already establishes existence there. We prove that every pair has a unique primitive contraction up to sequence equivalence. Consequently, a complete primitive catalogue determines every pair at every length.

The new structural result bounds the primitive length by $64$, $48$ and $64$ for distributions $(5,4)$, $(6,3)$ and $(7,2)$, respectively. It proves that no further primitive pairs occur beyond the finite ranges considered here. Enumeration within these bounds gives $91$ classes with two nonzero components. These are within the ranges of the earlier searches; the new point is the proof of completeness for all lengths. Together with the five established circulant classes, they give $96$ primitive representatives. Divisibility tests then decide existence without another sequence search. The catalogue also permits every class to be constructed or counted at a prescribed length.

The proof introduces an auxiliary integer coordinate for each nonzero entry. We show that complementarity forces this coordinate to be constant within each component. This prevents a nonzero rational solution of the associated homogeneous signed-difference equations after fixing one support position in each component at zero. A determinant argument then bounds the primitive length. \citet{LS2011} established finiteness results for circulant weighing matrices of fixed weight. Support-difference equations and determinant bounds occur in \citet[Theorem 13]{LS2018}. Here the essential additional step is the rigidity argument for two complementary components.

The following finite sets give the existence test. We call each a \emph{divisor base}: its members are admissible lengths, none divides another, and all admissible lengths are their multiples. For $(r,s)=(5,4),(6,3),(7,2),(9,0)$, respectively, put
\begin{equation}\label{eq:divisor-bases}
\begin{aligned}
 D_{54}&=\{5,7,8,11,12,13,17,19,23,29,41,43\},\\
 D_{63}&=\{7,10,11,12,13,15,17,19,23,25,29\},\\
 D_{72}&=\{10,12,13,14,15,19,23,25\},\\
 D_{90}&=\{13,24\}.
\end{aligned}
\end{equation}

\begin{theorem}\label{thm:main-spectrum}
Let $n$ be a positive integer and let $(r,s)\in\{(5,4),(6,3),(7,2),(9,0)\}$. A periodic complementary ternary pair of length $n$ and component weights $(r,s)$ exists if and only if $n$ is divisible by a member of $D_{rs}$ in~\eqref{eq:divisor-bases}. Consequently, a weight-nine pair of any distribution exists exactly when $n$ is divisible by a member of $D_{54}$.
\end{theorem}

The necessity combines the new length bound with exhaustive enumeration in the bounded range. The sufficiency uses the displayed constructions in Section~\ref{sec:seed-constructions}. The theorem therefore gives a decision test and a way to obtain a pair whenever the test succeeds.

An existence test does not describe every solution. The next theorem does so by identifying the complete set of starting pairs needed for uniform zero insertion.

\begin{theorem}\label{thm:complete-generation}
Let $(A,B)$ be a periodic complementary ternary pair of length $n$ and total weight nine, with components ordered by decreasing weight; $B$ may be zero. Under the sequence equivalence of Definition~\ref{def:equivalence}, $(A,B)$ is equivalent to a uniform zero insertion of exactly one of the $96$ primitive representatives in Appendices~\ref{app:catalogue} and~\ref{app:scalar-catalogue}. If that representative has length $m$, then $m\mid n$ and the insertion factor is $t=n/m$: one inserts $t-1$ zeros between consecutive entries, including cyclically. The value $t=1$ is allowed.

Conversely, every listed primitive representative whose length divides $n$ gives a pair of length $n$ in this way. The catalogue contains $59$, $20$, $12$ and $5$ representatives for distributions $(5,4)$, $(6,3)$, $(7,2)$ and $(9,0)$, respectively, with maximum primitive lengths $43$, $29$, $28$ and $26$. In particular, every weight-nine pair contracts to primitive length at most $43$.
\end{theorem}

The primitive class and its length are unique, although the transformations producing a particular pair need not be. The $33$ seeds in Section~\ref{sec:seed-constructions} suffice to construct one pair at every admissible length and distribution. The complete $96$-representative catalogue is needed to generate every pair. Section~\ref{sec:all-pair-generation} explains how to use it.

The following corollary uses the divisor bases to prove the implication conjecture of \citet[Conjecture 1]{GK2002}. This reduces questions about which weight distributions are available at a fixed length.
\begin{corollary}\label{cor:main-implications}
At every fixed length $n$, existence for distribution $(7,2)$ implies existence for $(6,3)$, and existence for $(6,3)$ implies existence for $(5,4)$.
\end{corollary}
\begin{proof}
Every member of $D_{72}$ is divisible by a member of $D_{63}$, and every member of $D_{63}$ is divisible by a member of $D_{54}$. Apply Theorem~\ref{thm:main-spectrum}.
\end{proof}

These implications concern existence at the same length. They do not give an entrywise conversion of an arbitrary pair from one distribution to another.

\section{Preliminaries and primitive contraction}\label{sec:prelim}
Write $C_n=\mathbb Z/n\mathbb Z$ for the cyclic group of positions. For a ternary sequence $A=(a_0,\ldots,a_{n-1})$, its support is $\supp(A)=\{j:a_j\ne0\}$. Its positive and negative describing sets are $A_+=\{j:a_j=1\}$ and $A_-=\{j:a_j=-1\}$, as in \citet{G2010}. Define
\[
 A(x)=\sum_{j=0}^{n-1}a_jx^j,\qquad A^*(x)=A(x^{-1}),\qquad
 P_A(s)=\sum_{j=0}^{n-1}a_ja_{j+s\bmod n}.
\]
Here $P_A$ is the PAF of $A$, and $\wt(A)=|\supp(A)|$. Polynomial multiplication is taken modulo $x^n-1$ in the group ring $\Z[C_n]$; thus exponents are reduced modulo $n$. The operation $*$ replaces each group element by its inverse. The pair $(A,B)$ has total weight nine and zero PAF if
\begin{equation}\label{eq:norm-nine}
 AA^*+BB^*=9\quad\text{in }\Z[C_n]=\Z[x]/(x^n-1).
\end{equation}
This says that $\wt(A)+\wt(B)=9$ and $P_A(s)+P_B(s)=0$ for $1\le s<n$.

Evaluation at one gives $A(1)^2+B(1)^2=9$. Up to component interchange and sign changes, the possible nonzero distributions and sign profiles are
\[
\begin{array}{c|cc}
\text{distribution}&(\# +,\# -)\text{ in }A&(\# +,\# -)\text{ in }B\\\hline
(5,4)&(4,1)&(2,2)\\
(6,3)&(3,3)&(3,0)\\
(7,2)&(5,2)&(1,1).
\end{array}
\]
Together with $(9,0)$, these are all possibilities~\citep{GK2002,G2010}.

The sequence operations in Definition~\ref{def:equivalence} preserve the PAF equations and both component weights.

If $C_A,C_B$ are the circulants whose successive rows are cyclic permutations of $A,B$, equation~\eqref{eq:norm-nine} is equivalent to $C_AC_A^{\mathsf T}+C_BC_B^{\mathsf T}=9I_n$. The standard two-circulant array~\citep[Theorem 4.46]{GS1979}
\begin{equation}\label{eq:weighing-array}
 W=\begin{pmatrix}C_A&C_B\\-C_B^{\mathsf T}&C_A^{\mathsf T}\end{pmatrix}
\end{equation}
is then a weighing matrix of order $2n$ and weight nine. All nonexistence statements in this paper concern the sequence pair, or equivalently this two-circulant construction, rather than arbitrary weighing matrices of order $2n$.
To recover the shortest pair underlying a given pair, one needs only its support differences. The following gcd test decides whether contraction is possible and preserves the PAF equations. Put
\[
 d=\gcd\bigl(n,\{u-v:u,v\in\operatorname{supp}(A)\},
                    \{u-v:u,v\in\operatorname{supp}(B)\}\bigr).
\]
Differences are represented by integers; their choice modulo $n$ does not change the greatest common divisor (gcd). Empty collections contribute nothing.
\begin{proposition}\label{prop:gcd-contraction}
After independent cyclic permutations of its nonzero components, a pair
is obtained by uniform zero insertion with factor $t>1$ if and only if
$t\mid d$. Its maximal contraction has length $n/d$, is primitive with
respect to zero insertion, and is unique up to independent cyclic
permutations. Complementarity, total weight, and the component
weights are preserved by both contraction and inflation.
\end{proposition}
\begin{proof}
Choose one reference support position in each nonzero component and cyclically permute that component so that the chosen position becomes zero.
All support positions are divisible by $t\mid n$ precisely when every
within-component difference is divisible by $t$. This is equivalent to
$t\mid d$. Divide the resulting indices by $d$ to obtain sequences on
$C_{n/d}$. The subgroup embedding $C_{n/d}\longrightarrow C_n$,
$j\mapsto dj$, induces an injective homomorphism of integral group rings,
so the complementarity identity is preserved in both directions. Choosing a different reference position
only cyclically permutes the contracted component. The gcd of the contracted
support differences and $n/d$ is one, proving primitiveness and maximality.
\end{proof}

\section{Rigidity in a cyclic group with a free coordinate}
To bound primitive lengths, we must rule out movement of support positions that preserves all matched differences. We represent this movement by an additional unrestricted integer coordinate. The result of this section forces that coordinate to be constant within each component.

Let $H$ be a finite cyclic group, written multiplicatively. We consider a ternary element $A\in\Z[H\times\Z]$. Write $A=\sum_k A_kz^k$, where $A_k\in\Z[H]$ collects the terms whose additional integer coordinate is $k$, with that coordinate suppressed. Only finitely many $A_k$ are nonzero. The integer $k$ is the \emph{free coordinate}. Each $A_k$ is the group-ring coefficient of $z^k$ and records the positions in $H$ and the signs of the entries at that coordinate. We use the analogous notation $B=\sum_k B_kz^k$. The operation $*$ inverts both coordinates, and all integer coefficients are $0,\pm1$.

For a ternary $C\in\Z[H]$, its weight counts its nonzero integer coefficients, and its \emph{augmentation} $C(1)$ is their sum. Multiplication by an element of $H$ cyclically permutes these coefficients. We call $C$ \emph{balanced} if $C(1)=0$, meaning that it has equal numbers of positive and negative entries. Augmentation over $H$ gives $\overline A(z)=\sum_k A_k(1)z^k$ and $\overline B(z)=\sum_k B_k(1)z^k$.

The next lemma restricts the coefficient sums $A_k(1)$ and $B_k(1)$. Since the total weight is nine, only a few possibilities remain for the numbers of positive and negative entries in the nonzero $A_k$ and $B_k$.

\begin{lemma}\label{lem:augmentation-free}
If integer Laurent polynomials $F,G$ satisfy $FF^*+GG^*=9$, one is zero and the other is $\pm3z^a$. Consequently, if ternary group-ring elements $A,B$ have total weight nine and $AA^*+BB^*=9$, then, after sign changes, interchange, and independent shifts in the free coordinate, their augmentations over $H$ are $\overline A(z)=3$ and $\overline B(z)=0$.
\end{lemma}
\begin{proof}
Reduce the integer identity modulo three, in the field $\mathbb F_3$ of three elements. If a reduction is nonzero, let $r$ be the smaller order of vanishing of the two reductions at $z=1$, assigning infinite order to the zero polynomial. Write $z=1+t$, and let $a,b$ be the coefficients at order $r$; at least one is nonzero. Since $(1+t)^{-1}-1=-t+O(t^2)$, the coefficient of $t^{2r}$ in the norm sum is $(-1)^r(a^2+b^2)\ne0$ in $\mathbb F_3$, contradicting its zero right-hand side. Thus both reductions vanish. Write $F=3F_0,G=3G_0$; then
$F_0F_0^*+G_0G_0^*=1$. Its constant coefficient is a sum of integer
squares, so exactly one coefficient of $F_0,G_0$ is nonzero, and equals
$1$ or $-1$.
\end{proof}

Two elementary obstructions will restrict cancellation between terms at different free coordinates. A character of $H$ is a homomorphism into the complex numbers of absolute value one. It is \emph{faithful} if different group elements have different character values, that is, if the character is injective.

\begin{lemma}\label{lem:binomial-facts}
Two nonzero binomials of the form $u-v$ in $\mathbb Z[H]$ have nonzero
product. Moreover, if $C(1-h)=0$, then multiplication by $h$ leaves $C$
unchanged. Consequently, the order of $h$ divides the number of positive
coefficients and the number of negative coefficients of a ternary $C$.
\end{lemma}
\begin{proof}
A faithful complex character of the cyclic group $H$ is nonzero on every
nonzero difference $u-v$, proving the first assertion. The second is the
coefficient equality $C=Ch$; nonzero coefficients are constant on the
orbits of the subgroup generated by $h$.
\end{proof}

\begin{lemma}\label{lem:four-impossible}
Let $h\in H$ have order three. There is no ternary $Y\in\mathbb Z[H]$
with two positive and two negative coefficients such that
\[
 YY^*=4-2h-2h^2.
\]
\end{lemma}
\begin{proof}
Put $K=\langle h\rangle$. Compression to $H/K$, obtained by adding coefficients in each coset of $K$, gives
$\overline Y\,\overline Y^*=0$. Its identity coefficient is a sum of
integer squares, so $\overline Y=0$. A nonempty balanced ternary subset
of a coset of $K$ has exactly two entries. Hence the four entries of $Y$
occupy two distinct cosets, and, after multiplying $Y$ by a suitable element of $H$,
$Y=f+tg$, where $f,g\in\mathbb Z[K]$ are nonzero signed binomials and
$t\notin K$.
If $tK\ne t^{-1}K$, the two cross terms of $YY^*$ occupy distinct
nonidentity quotient cosets and cannot vanish, by
Lemma~\ref{lem:binomial-facts}. If $tK=t^{-1}K$, then $t^2\in K$.
Replacing $t$ by $th^j$ and adjusting $g$, choose $t^2=1$, since
multiplication by two is invertible modulo three. The cross coefficient
in the coset $tK$ is then $fg^*+gf^*$.
At a character with $h\mapsto\omega=e^{2\pi i/3}$, every nonzero
signed binomial in $\mathbb Z[K]$ has value $(1-\omega)$ times a sixth
root of unity. Thus $f(\omega)\overline{g(\omega)}$ is three times a
sixth root of unity, whose real part is never zero. Therefore
$fg^*+gf^*\ne0$, again a contradiction.
\end{proof}

The following result shows that exactly one $A_k$ and exactly one $B_k$ are nonzero. It will force the signed-difference equations in the next section to have full rank.

\begin{theorem}\label{thm:free-rigidity}
Suppose $A,B\in\mathbb Z[H\times\mathbb Z]$ are a ternary complementary
pair of total weight nine, and both components are nonzero. Every support
position of $A$ has the same free coordinate, and the same is true
separately of $B$.
\end{theorem}
\begin{proof}
Normalise the row sums, which are the sums of the sequence entries, as in Lemma~\ref{lem:augmentation-free}.
The component of row sum three has exactly one group-ring coefficient of augmentation
three. All other group-ring coefficients in both components have augmentation
zero. Every nonzero ternary element of augmentation zero has even weight,
at least two.

\medskip\noindent\textbf{Distribution $(5,4)$.}
Write $A$ for the row-sum-three component. Either exactly one $A_k$ is nonzero,
of weight five, or the nonzero $A_k$ are a positive three-term element $P$
and a balanced binomial $D$. Similarly, the nonzero $B_k$ are either one
four-term element or two balanced binomials $E,F$.
The extreme indices of a component are the smallest and largest $k$ for which its group-ring coefficient is nonzero. If only $B$ has two nonzero $B_k$, their nonzero cross product contradicts
complementarity. If only $A$ has two nonzero $A_k$, then $PD^*=0$. By
Lemma~\ref{lem:binomial-facts}, the difference $h$ in $D$ has order three
and, after multiplication by suitable elements of $H$ and sign changes, $P=1+h+h^2$, $D=1-h$.
The coefficient of $z^0$ of the complementarity identity would then
require $BB^*=4-2h-2h^2$, contradicting
Lemma~\ref{lem:four-impossible}.

It remains to consider two nonzero $A_k$ and two nonzero $B_k$. The product $EF^*$
is nonzero, so $PD^*$ is also nonzero, and the two free-coordinate gaps
must agree up to sign. The constant coefficient gives
\[
 PP^*+DD^*+EE^*+FF^*=9.
\]
Consequently the three unordered differences of the three positive
positions of $P$ are exactly the three binomial differences of $D,E,F$,
counted with multiplicity modulo inversion. This conclusion also holds
for an element of order two: its contribution from a pair of distinct positions is twice the
corresponding group element.
Choose a generator $x$ of $H$, with exponents taken modulo $|H|$. Cyclically permute the coefficients of $P$ and label its support positions so that $P=1+x^a+x^b$, the difference
in $D$ is $a$, and the other two differences are $b$ and $b-a$.
The vanishing cross coefficient, after choosing orientations and absorbing signs and powers of $x$, becomes
\begin{equation}\label{eq:triangle-binomials}
 (1+x^a+x^b)(1-x^a)
   =\varepsilon x^q(1-x^b)(1-x^{b-a}),\qquad \varepsilon\in\{1,-1\}.
\end{equation}
Here $a,b$ are nonzero and distinct in the cyclic group $H$.
The group-valued first moment $\mu(\sum c_jx^j)=\sum c_jj$ satisfies
$\mu(UV)=U(1)\mu(V)+V(1)\mu(U)$.
Applying it to~\eqref{eq:triangle-binomials} gives $3a=0$.
Since $a\ne0$, it has order three. Passing to $H/\langle a\rangle$
in~\eqref{eq:triangle-binomials} gives $(1-y)^2=0$, where $y$ is the
image of $x^b$. Evaluation at all complex characters forces $y=1$.
Thus $b\in\langle a\rangle$, and distinctness forces $b=2a$.
The left side of~\eqref{eq:triangle-binomials} is now zero, whereas its
right side is a product of two nonzero binomials. This is impossible.

\medskip\noindent\textbf{Distribution $(6,3)$.}
Exactly one $B_k$ is nonzero and has three positive entries. Every $A_k$
has coefficient sum zero. If more than one $A_k$ is nonzero, their weights are either
$2+2+2$ or $2+4$. In the case $2+2+2$, the $A_k$ at the extreme indices are binomials and
have nonzero cross product, contradicting complementarity.
For $2+4$, write the two nonzero $A_k$ as $D=1-h$ and $C$, after multiplying the binomial by a suitable element of $H$
and changing its sign. Their cross product vanishes, so $C$ is invariant under
$h$. Its two positive and two negative entries force $h$ to have order
two. At the nonidentity shift $h$, the coefficient of $CC^*$ is four,
and that of $DD^*$ is minus two. The component $B$ has only positive
entries, so its autocorrelation coefficient at $h$ is nonnegative. The
summed coefficient is therefore at least two, a contradiction.

\medskip\noindent\textbf{Distribution $(7,2)$.}
Exactly one $B_k$ is nonzero, and it is a balanced binomial. If more than one
$A_k$ is nonzero, their weights are $5+2$, $3+2+2$, or $3+4$.
For $5+2$, the vanishing cross product makes the five-term $A_k$ invariant under
the nonzero difference in the binomial. Its augmentation is three,
so the order of that difference divides both five and three, which is
impossible.
For $3+2+2$, if the $A_k$ at both extreme indices are binomials, their cross product is
nonzero, a contradiction. Otherwise the positive three-term element $P$
occurs at an extreme index. If $D$ is the binomial at the opposite extreme index, then $PD^*=0$. Thus
$P=1+h+h^2$ after multiplication by a suitable element of $H$, with $h$ of order three.
The coefficient of $z^0$ shows that all three binomial differences
(from the other two nonzero $A_k$ and the nonzero $B_k$) are $h$ or $h^2$, because they must
match the three differences of $P$. Both products with $P$ vanish.
However, the cross product between the two binomial $A_k$ is nonzero
and has no remaining term to cancel it.

It remains to consider $3+4$. Write the two nonzero $A_k$ as a positive three-term
element $P=1+x^a+x^b$ and a balanced four-term element $C$. They satisfy $PC^*=0$.
For every complex character $\chi$ with $\chi(C)\ne0$, the three unit
complex numbers $1,\chi(x^a),\chi(x^b)$ sum to zero. They are therefore
the three cube roots of unity. In particular
$\chi(x^{3a})=\chi(x^{a+b})=1$.
Put $L=\langle x^{3a},x^{a+b}\rangle$. For each $g\in L$, every character vanishes on $(g-1)C$. Characters distinguish elements of the complex group algebra, so $(g-1)C=0$. Thus $C$ is invariant under $L$. Its two positive and two negative entries imply
$|L|\mid2$. If $|L|=2$, its nonidentity element $u$ has coefficient four
in $CC^*$. The coefficient in $PP^*$ is nonnegative and that in $BB^*$
is at least minus two, contradicting the coefficient of $z^0$ in the complementarity identity.
Thus $L$ is trivial. Since $0,a,b$ are distinct, $x^a=h$ has order three,
$b=-a$, and $P=1+h+h^2$.
Compressing the coefficient of $z^0$ in the complementarity identity to $H/\langle h\rangle$
gives $\overline C\,\overline C^*+\overline B\,\overline B^*=0$.
Its identity coefficient forces both compressed elements to vanish.
Hence the difference in $B$ is $h$ or $h^2$. The uncompressed identity
now gives $CC^*=4-2h-2h^2$, contrary to
Lemma~\ref{lem:four-impossible}. Hence exactly one $A_k$ and exactly one $B_k$ are nonzero, as required.
\end{proof}

\section{Full rank of every valid matching system}
Each cancellation of signed differences gives an integer equation in the support positions. We will bound determinants of these equations to bound the primitive length. Label the nine nonzero entries separately in their two components, even when their cyclic positions coincide. An \emph{edge} is an unordered pair of distinct labelled entries in one component; its sign is their product. Choosing a direction orders its two entries. For the standard basis rows $e_i$, the \emph{incidence row} of the directed edge $i\to j$ is $e_i-e_j$. A directed incidence matrix has these rows and one column for each labelled entry. A \emph{minor of order $r$} is the determinant of a square submatrix obtained by selecting $r$ rows and $r$ columns.
\begin{lemma}\label{lem:incidence-bound}
Every square submatrix of a directed incidence matrix has determinant $0$, $1$, or $-1$.
Consequently, every $r\times r$ submatrix of a matrix whose rows are differences of two incidence rows has determinant of absolute value at most $2^r$.
\end{lemma}
\begin{proof}
Let $Q$ be a square submatrix of a directed incidence matrix. If a row is zero, then $\det Q=0$. If a row has one nonzero entry, that entry is $1$ or $-1$, and expansion along the row reduces the claim to a smaller square submatrix. Induction proves the assertion in this case. Otherwise, every row contains one $1$ and one $-1$, so every row sums to zero and $Q$ is singular. For the second assertion, multilinearity expands the determinant into $2^r$ signed terms. Each is the determinant of incidence rows restricted to the selected columns and has absolute value at most one by the same argument.
\end{proof}

\begin{lemma}\label{lem:unordered-all-lengths}
For a periodically complementary ternary pair of total weight nine, positive and negative unordered support differences can be matched modulo inversion at any length. Each match can be oriented as an equality of directed differences.
\end{lemma}
\begin{proof}
An unordered pair of distinct support positions $i,j$ in one component contributes its sign product at both differences $d=j-i$ and $-d$, taken modulo $n$. These differences are distinct unless $n$ is even and $d=n/2$. At that shift the two contributions coincide, so each pair is counted twice. The zero-PAF condition therefore gives equal numbers of positive and negative pairs with each difference modulo inversion, also at $n/2$. Match pairs of opposite signs and choose their directions so that the directed differences agree.
\end{proof}

Assume both components are nonzero. Choose one reference support position in each component and cyclically permute that component to put the chosen position at zero. The other seven positions form a vector $v\in(\Z/n\Z)^7$. For each matching in Lemma~\ref{lem:unordered-all-lengths}, write the equality of directed differences as an integer homogeneous equation. The resulting matrix $M$ has rows that are differences of two incidence rows, with the columns of the two reference positions removed, and $Mv=0\pmod n$. We call a matching valid when it arises from a complementary pair.

Full rank is the link between the preceding rigidity result and a length bound. Here rank is taken over $\Q$, the field of rational numbers. Rank seven means that the seven columns of $M$ are linearly independent over $\Q$. For a primitive pair, the length must divide a nonzero determinant of these equations.

\begin{theorem}\label{thm:matching-full-rank}
Every such valid matching matrix has rank seven over $\Q$. If the pair is primitive, its length divides every nonzero seven-by-seven minor of $M$.
\end{theorem}
\begin{proof}
Suppose its rank over $\Q$ were smaller than seven. A nonzero rational solution of $Mu=0$ can be multiplied by a common denominator to give a nonzero integer vector $u$. Give the two reference entries free coordinate zero and lift each remaining support position to $(v_j,u_j)\in C_n\times\Z$. Positions remain distinct within each component because their first coordinates were distinct. Every matched difference is equal in both coordinates, so the lifted ternary pair still satisfies the complementarity identity of total weight nine. Theorem~\ref{thm:free-rigidity} forces all free coordinates within each component to equal the coordinate of its reference entry, which is zero. This contradicts $u\ne0$.

Now let $Q$ be any nonsingular seven-row submatrix, with determinant $D$. The adjugate identity and $Qv=0\pmod n$ give $Dv_j=0\pmod n$ for every $j$. Choose integer representatives for the $v_j$. Primitiveness gives $\gcd(n,v_1,\ldots,v_7)=1$, by Proposition~\ref{prop:gcd-contraction}. B\'ezout's identity supplies integers $c_0,\ldots,c_7$ such that
\[
 c_0n+\sum_{j=1}^7 c_jv_j=1.
\]
Multiplying by $D$ gives $D=c_0nD+\sum_{j=1}^7c_jDv_j$. Every term on the right is divisible by $n$, so $n\mid D$.
\end{proof}

By Theorem~\ref{thm:matching-full-rank}, the length $n$ of a primitive pair divides a nonzero determinant $D$ of a $7\times7$ submatrix of $M$. Lemma~\ref{lem:incidence-bound} gives $|D|\le2^7=128$, so $n\le128$. The next section uses the sign patterns of the three weight distributions $(5,4)$, $(6,3)$ and $(7,2)$ to improve these bounds to $64$, $48$ and $64$, respectively. Consequently, every primitive pair with two nonzero components has length at most $64$.

\section{Explicit primitive-length bounds}
The weight distributions give different patterns of signed differences. We use these patterns to find a nonzero minor of size seven with a smaller determinant bound in each case.
\subsection{Distribution (5,4)}
There are eight positive edges: six on the positive $A$ positions and two joining equal signs in $B$. There are four negative edges in each component. Classify a matching row as $AA$, $BB$, $AB$ or $BA$ by the components of its positive and negative edges. If $\ell=\#BA$, the row counts are
\[
(\#AA,\#BB,\#AB,\#BA)=(4-\ell,2-\ell,2+\ell,\ell),\quad \ell=0,1,2.
\]
Deleting the columns of the two reference positions leaves four $A$ columns and three $B$ columns.
A selected seven-row determinant with $r_A$ internal $AA$ rows, $r_B$ internal $BB$ rows and $c$ crossing rows has at most
\[
2^{r_A+r_B}\binom{c}{4-r_A}
\]
potentially nonzero incidence terms: precisely four chosen incidence rows must lie in $A$ and three in $B$. Each term has absolute value at most one. Deleting one row from the eight gives
\[
\begin{array}{c|ccc}
\ell&\text{delete }AA&\text{delete }BB&\text{delete crossing}\\\hline
0&64&32&64\\
1&48&32&48\\
2&40&\text{not present}&40.
\end{array}
\]
Every seven-row minor is therefore at most $64$ in absolute value. Some is nonzero by Theorem~\ref{thm:matching-full-rank}, so a primitive pair has length at most $64$.

\subsection{Distribution (6,3)}
The six positive edges of $A$ form two triangles, and the three positive edges of $B$ form one triangle. All nine negative edges are in $A$, joining its positive and negative triples.
Choose two matching rows whose $B$ edges form a spanning tree of its triangle. Their two remaining $B$ columns give a determinant $1$ or $-1$. Use those rows to eliminate the $B$ columns from the third such row. The three oriented triangle edges have a relation with coefficients $1$ or $-1$, so the residual row in the five $A$ columns is a signed sum of three negative incidence rows. The other six residual rows are differences of two $A$ incidence rows.

The residual matrix has rank five because the full matrix has rank seven. Select five independent residual rows. If all five are ordinary rows, its determinant has at most $2^5=32$ incidence terms. If the exceptional row occurs, it has at most $3\cdot2^4=48$ terms. Each term has absolute value at most one. Restoring the two unit pivot rows gives a nonzero seven-by-seven minor of the original matrix with the same absolute value: subtracting multiples of pivot rows does not change that determinant. Hence a primitive $(6,3)$ pair has length at most $48$.

\subsection{Distribution (7,2)}
Write $B=1-x^d$. There are eleven positive edges, all in $A$, and eleven negative edges, ten in $A$ and one in $B$. Exactly one row of the matching matrix contains the $d$ column, with coefficient $1$ or $-1$. The other ten rows involve only the six remaining $A$ columns. Full rank implies that six of those ten rows are independent. Select them and the row containing $d$. Expanding in the $d$ column leaves a six-by-six determinant of incidence differences, bounded by $2^6=64$. Thus a primitive $(7,2)$ pair has length at most $64$.

These bounds depend only on the weight distribution. They give finite ranges containing every primitive pair, so exhaustive classification in those ranges settles the problem at all lengths.

\begin{theorem}\label{thm:uniform64}
Every periodic complementary ternary pair of total weight nine with both components nonzero is obtained, up to independent cyclic permutations of the components, by uniform zero insertion from a primitive pair of length at most $64$. More precisely, the bounds for $(5,4),(6,3),(7,2)$ are $64,48,64$, respectively.
\end{theorem}
\begin{proof}
Apply Proposition~\ref{prop:gcd-contraction} to remove uniformly inserted zeros after independent cyclic permutations, obtaining a primitive pair. Theorem~\ref{thm:matching-full-rank} and the determinant bounds above then give the stated bounds on its length. Uniform zero insertion recovers the original pair up to independent cyclic permutations.
\end{proof}

\section{Finite classification and proof of the existence theorem}\label{sec:finite}
Theorem~\ref{thm:uniform64} reduces the three nonzero distributions to primitive lengths at most $64$, $48$ and $64$. We enumerate every candidate in those ranges up to the reductions described below. All computations use integer signed-difference counts. The resulting primitive counts reproduce Tables 2--4 of~\citet{G2010}; the enumeration is also supplied independently so that the finite part of the classification can be repeated.

\subsection{Exhaustive candidate generation}\label{sec:enumeration}
For a signed support, define the folded difference vector $c=(c_1,\ldots,c_{\lfloor n/2\rfloor})$ by adding the sign product of each unordered support pair to the coordinate $\min(d,n-d)$, where $d$ is its difference represented in $\{1,\ldots,n-1\}$. When $n$ is even, the entry corresponding to shift $n/2$ equals half the PAF value at that shift. Equality and cancellation of these vectors are therefore equivalent to equality and cancellation of the PAF values at nonzero shifts, including at even lengths. Storing the PAF coefficients directly gives the same test.

For distribution $(5,4)$, cyclically permute $A$ so that its unique negative entry occurs at position zero. Choose its four positive positions from $\{1,\ldots,n-1\}$. Retain one representative under common unit multiplication. Independently normalise $B$ under cyclic permutations, reversal and sign change, and store its folded vector. Match it against the negative folded vector of $A$. Since the independent operations on $B$ preserve its PAF, every equivalence class is represented by this procedure.

For $(6,3)$, cyclically permute the balanced component $A$ so that a positive entry occurs at position zero, and write its other two positive positions as $a,b$. Common unit multiplication can map either nonzero gap to its gcd with $n$. Thus we may require at least one of $a,b$ to be a proper divisor of $n$. Choose the three negative positions from the remaining residues. Independently cyclically permute $B$ to have positive support $\{0,c,d\}$ and match its folded vector with the negative vector of $A$.

For $(7,2)$, cyclically permute $A$ so that one negative entry occurs at position zero, and use a common unit to make its other negative position a proper divisor of $n$. Enumerate every five-element positive support disjoint from the two negative positions. A folded vector can be completed by $B$ exactly when it is zero except for a single entry equal to one. If that entry is at $d$, choose $B=1-x^d$. Reflection of $B$ permits either orientation of the gap.

The unit reductions are exact because the units modulo $n$ act transitively on the residues having the same gcd with $n$. Indeed, write $a=ga_0$, $n=gn_0$, with $\gcd(a_0,n_0)=1$. The inverse of $a_0$ modulo $n_0$ lifts to a unit modulo $n$, and multiplication by that lift sends $a$ to $g$. A unit modulo a divisor of $n$ always has such a lift, by the Chinese remainder theorem.

In every case we retain primitive pairs using Proposition~\ref{prop:gcd-contraction} and remove duplicates by exact canonicalisation under Definition~\ref{def:equivalence}. Canonicalisation considers every common unit multiplier and minimises each component lexicographically under its independent cyclic permutations, reversals and allowed sign changes. Every retained pair is also checked directly against all $n$ PAF equations. The programs neither stop after reaching an expected count nor omit a candidate on the basis of the published results.

The length bounds make this calculation exhaustive for the original unbounded problem. Let $P_{rs}(m)$ denote the number of primitive sequence-equivalence classes of distribution $(r,s)$ and length $m$.

\begin{proposition}\label{prop:finite-enumeration}
For the three distributions with both components nonzero, the exhaustive procedures above give precisely the primitive sequence-equivalence classes listed in Table~\ref{tab:primitive-counts}. There are $59$, $20$ and $12$ classes for $(5,4)$, $(6,3)$ and $(7,2)$, respectively. Every other length within the applicable theoretical bound has zero primitive classes.
\end{proposition}
\begin{proof}
The normalisations just proved preserve at least one representative of every equivalence class. After normalisation, the programs enumerate every permitted choice of the remaining support positions and test the exact signed-difference equations. Canonicalisation identifies precisely the equivalence relation of Definition~\ref{def:equivalence}. Running these finite procedures for $1\le n\le64$, $6\le n\le48$, and $7\le n\le64$, respectively, gives the counts in Table~\ref{tab:primitive-counts}; smaller lengths for the last two distributions are impossible by the larger component weight. Online Resource~1 (\texttt{ESM\_1.zip}) contains the source code, results for each length and all retained representatives. Its single \texttt{readme.txt} explains every file and how to reproduce this finite calculation. Each count agrees with the published exhaustive data~\citep{G2010} after primitive contraction.
\end{proof}

\begin{table}[htbp]\centering\small
\caption{Primitive class counts $P_{rs}(m)$ under Definition~\ref{def:equivalence}. Unlisted lengths have count zero. The scalar row is established in Proposition~\ref{prop:scalar-five}.}\label{tab:primitive-counts}
\begin{tabular}{clrr}\toprule
Distribution&$(m,P_{rs}(m))$&Total&Maximum $m$\\\midrule
$(5,4)$&\begin{tabular}[c]{@{}l@{}}
$(5,1),(7,1),(8,3),(10,3),(11,5),(12,2),(13,5),$\\
$(14,6),(15,4),(16,2),(17,3),(19,4),(20,5),$\\
$(21,2),(22,3),(23,2),(25,2),(28,3),(29,1),(41,1),(43,1)$
\end{tabular}&59&43\\[4pt]
$(6,3)$&\begin{tabular}[c]{@{}l@{}}
$(7,1),(10,1),(11,1),(12,1),(13,3),(15,3),(17,1),$\\
$(19,3),(20,1),(21,2),(23,1),(25,1),(29,1)$
\end{tabular}&20&29\\[4pt]
$(7,2)$&\begin{tabular}[c]{@{}l@{}}
$(10,1),(12,3),(13,1),(14,1),(15,1),$\\
$(19,1),(23,1),(25,1),(28,2)$
\end{tabular}&12&28\\[4pt]
$(9,0)$&$(13,2),(24,1),(26,2)$&5&26\\\bottomrule
\end{tabular}\end{table}

\begin{proof}[Proof of Theorem~\ref{thm:main-spectrum}]
Suppose both components are nonzero. By Theorem~\ref{thm:uniform64}, the pair is a uniform zero insertion of a primitive pair whose length $m$ divides $n$ and lies within the corresponding theoretical bound. Proposition~\ref{prop:finite-enumeration} gives all possibilities for $m$. The minimal members under divisibility of the three lists in Table~\ref{tab:primitive-counts} are exactly $D_{54},D_{63},D_{72}$. Hence some member of the applicable divisor base divides $n$.

Conversely, Section~\ref{sec:seed-constructions} supplies a pair at every length in each divisor base. Uniform zero insertion gives every multiple. For $(9,0)$, necessity is the known circulant classification~\citep{AAMS2008}, and explicit examples at $13$ and $24$ are also included below. Finally, every member of $D_{63}$ or $D_{72}$ is divisible by a member of $D_{54}$; the $(9,0)$ lengths are already covered by $13$ and $12$ in $D_{54}$. This proves the last assertion as well.
\end{proof}

\subsection{The established scalar case}\label{sec:scalar}
The length bound requires two nonzero components. Known circulant results cover the remaining distribution. Recording their representatives under the same sequence equivalence makes the catalogue and counting formula apply to every weight distribution.
When $B=0$, the condition reduces to $P_A(s)=0$ at every nonzero cyclic shift, and the identity $AA^*=9$ defines a circulant weighing matrix $CW(n,9)$. For a single sequence, Definition~\ref{def:equivalence} becomes $A(x)\mapsto\varepsilon x^aA(x^u)$, where $\varepsilon\in\{1,-1\}$ and $\gcd(u,n)=1$; reversal is already included by $u=-1$. This is sequence equivalence, rather than unrestricted signed row and column equivalence of matrices.

\citet[Theorem 3]{AAMS2008} prove that proper circulant weighing matrices of weight nine exist exactly at orders $13$, $24$ and $26$. Their definition of proper agrees with primitiveness in Definition~\ref{def:equivalence}. We incorporate this known result and the uniqueness result at order $24$ from~\citet{SS2013}. The following short orbit calculation records all scalar representatives in the equivalence convention of this paper; it is included for completeness and is not claimed as a new scalar classification.

\begin{proposition}\label{prop:scalar-five}
The distribution $(9,0)$ has exactly five primitive sequence-equivalence classes: two at length $13$, one at length $24$, and two at length $26$. Their signed supports are listed in Appendix~\ref{app:scalar-catalogue}.
\end{proposition}
\begin{proof}
By \citet[Theorem 3]{AAMS2008}, a primitive pair of distribution $(9,0)$ can have length only $13$, $24$ or $26$.
We give a short explicit count for $n=13,26$. Change the overall sign
so that the row sum is three. There are six positive and three negative
entries. A global sign change cannot identify two distinct representatives whose row sums are both three. By the multiplier theorem~\citep[Lemma 6]{AAMS2008}, a
cyclic permutation gives a row $D$ with $D(x^3)=D(x)$, since $3$ is coprime to $n$.
The multiplication-by-three orbits have sizes one and three. There is
one fixed point for $n=13$ and two for $n=26$. It follows that the
negative support is one three-element orbit, while the positive support
is the union of two distinct three-element orbits. Fixed points cannot
occur: neither $3$ nor $6$ can be expressed using one or two fixed
points and a sum of orbit sizes three.

For $n=13$, write
\[
 C_0=\{1,3,9\},\quad C_1=\{2,5,6\},\quad
 C_2=\{4,10,12\},\quad C_3=\{7,8,11\}.
\]
Unit multiplication is transitive on these four orbits, so normalise the
negative support to $C_0$. Of the three choices of two positive orbits,
$C_1\cup C_3$ and $C_2\cup C_3$ give sequences whose PAF vanishes at every nonzero shift.
The remaining choice $C_1\cup C_2$ has PAF $-2$ at shift one.
These two solutions are inequivalent. Indeed, any map $j\mapsto uj+a\pmod{13}$, with $\gcd(u,13)=1$,
that preserves $C_0$ has $3a=0$ modulo $13$, by summing its elements, and
hence $a=0$. Its multiplier lies in $\{1,3,9\}$ and fixes each $C_i$.
Thus there are exactly two classes at length $13$, both primitive.

For $n=26$, put $E_i=2C_i$ and $O_i=E_i+13$ modulo $26$.
The eight triples are $E_i$ and $O_i$ for $0\le i\le3$, while
$0,13$ are fixed points. If the negative triple is odd, cyclically permute the row by
$13$ positions; then use a unit multiplier to normalise the negative support to
$E_0=\{2,6,18\}$. The following complete table tests the $\binom72=21$
possible pairs of positive orbits. An entry $(s,c)$ means that the
PAF at shift $s$ is the nonzero integer $c$; a zero entry means that
the PAF vanishes at every nonzero shift.
\[
\begin{array}{c|rrrrrr}
 &E_2&E_3&O_0&O_1&O_2&O_3\\\hline
 E_1&(2,-2)&0&(1,-1)&(1,-1)&(1,1)&0\\
 E_2&&0&(1,-1)&(1,1)&(1,1)&(1,-1)\\
 E_3&&&(5,1)&(1,1)&0&(1,-2)\\
 O_0&&&&(1,-3)&(1,-3)&(1,-4)\\
 O_1&&&&&(1,-2)&(1,-3)\\
 O_2&&&&&&(1,-3)
\end{array}
\]
Each entry follows by expanding $DD^*$ using the displayed orbit sets;
the four zero entries are exact group-ring identities $DD^*=9$.
The choices $E_1\cup E_3$ and $E_2\cup E_3$ have all support indices
even and are inflations of the two length-$13$ classes. The choices
$E_1\cup O_3$ and $E_3\cup O_2$ are primitive, since their support
differences have gcd one with $26$.
As before, any map $j\mapsto uj+a\pmod{26}$, with $\gcd(u,26)=1$, that preserves $E_0$ must satisfy $3a=0$ modulo $26$, hence
$a=0$. Its multiplier lies in $\{1,3,9\}$ and fixes every $E_i$
and $O_i$. All four valid normal forms are therefore inequivalent.
In particular, there are exactly two primitive classes at length $26$.

\citet[Theorem 5.1]{SS2013} prove that every $CW(24,9)$ is equivalent
under the sequence operations of Definition~\ref{def:equivalence} to
\[
 -1+(1-\delta^4)(\delta+\delta^3)
       +(\alpha+\alpha^2)(1+\delta^4),
 \qquad \alpha^3=\delta^8=1.
\]
Here $\alpha$ and $\delta$ have orders three and eight, respectively. Taking $\alpha=x^8$ and $\delta=x^3$ gives positive support
$\{3,4,8,9,16,20\}$ and negative support $\{0,15,21\}$.
Thus the scalar case has five primitive sequence-equivalence
classes in all: two at $13$, one at $24$, and two at $26$.

\end{proof}

The programs in Online Resource~1 select one representative from each class of three-element negative supports under cyclic permutations and position multipliers coprime to $n$. For each representative, they test every disjoint six-element positive support. This exhaustive check gives the same five primitive classes at $13$, $24$ and $26$. It checks the listed supports and counts; the restriction to those three primitive lengths follows from the cited theoretical result.

\begin{proof}[Proof of Theorem~\ref{thm:complete-generation}]
By Proposition~\ref{prop:gcd-contraction}, independent cyclic permutations followed by maximal contraction give a primitive pair of length $m\mid n$. For two nonzero components, Theorem~\ref{thm:uniform64} puts $m$ within the relevant finite bound, and Proposition~\ref{prop:finite-enumeration} identifies its class with exactly one representative in Appendix~\ref{app:catalogue}. If $B=0$, Proposition~\ref{prop:scalar-five} instead identifies exactly one of the five representatives in Appendix~\ref{app:scalar-catalogue}. Every equivalence at length $m$ lifts to an equivalence after inflation to length $n$: a cyclic permutation by $a$ positions lifts to one by $(n/m)a$ positions, reversals and sign changes lift directly, and a common unit modulo $m$ lifts to a unit modulo $n$ by the Chinese remainder theorem. Hence the original pair is equivalent to the inflation of that representative. Conversely, every such inflation satisfies the complementarity identity by Proposition~\ref{prop:gcd-contraction}.

For uniqueness, the maximal contraction length is invariant under all operations in Definition~\ref{def:equivalence}. An equivalence between two inflated pairs of the same contraction length induces an equivalence between their primitive contractions: the common unit reduces modulo $m$, and the independent component operations descend after independently cyclically permuting the components to put their chosen reference positions at zero. Two different primitive lengths therefore cannot represent the same class at length $n$, nor can two inequivalent primitive representatives of the same length. The counts and maxima follow from Table~\ref{tab:primitive-counts}.
\end{proof}

\subsection{Counting all equivalence classes}\label{sec:counting}
The unique primitive contraction also permits counting without generating the longer sequences. Each primitive class contributes one class at every multiple of its length. Let $N_{rs}(n)$ denote the number of sequence-equivalence classes of distribution $(r,s)$ at length $n$, and use the primitive counts $P_{rs}(m)$ from Table~\ref{tab:primitive-counts}.
\begin{theorem}\label{thm:counting}
For $(r,s)\in\{(5,4),(6,3),(7,2),(9,0)\}$ and every positive integer $n$,
\[
 N_{rs}(n)=\sum_{m\mid n}P_{rs}(m),\qquad
 P_{rs}(n)=\sum_{m\mid n}\mu(n/m)N_{rs}(m),
\]
where the M\"obius function satisfies $\mu(1)=1$, is zero on integers divisible by a prime square, and equals $(-1)^k$ on products of $k$ distinct primes. In particular, the four counting functions have finite support in their primitive form and are determined for every length by Table~\ref{tab:primitive-counts}.
\end{theorem}
\begin{proof}
Theorem~\ref{thm:complete-generation} gives a bijection between classes at length $n$ and primitive classes whose lengths divide $n$. Summing their counts gives the first identity. M\"obius inversion gives the second.
\end{proof}

In particular, writing $\mathbf1_{d\mid n}=1$ when $d\mid n$ and zero otherwise, the scalar count is
\begin{equation}\label{eq:scalar-count}
 N_{90}(n)=2\mathbf1_{13\mid n}+\mathbf1_{24\mid n}+2\mathbf1_{26\mid n}.
\end{equation}
Thus $N_{90}(26)=4$: two classes are zero insertions of length-$13$ representatives and two are primitive at length $26$. The total number of weight-nine sequence-equivalence classes at any length is
\[
 N_9(n)=N_{54}(n)+N_{63}(n)+N_{72}(n)+N_{90}(n),
\]
with every term explicitly determined by Table~\ref{tab:primitive-counts}.

As an independent consistency check, the formula for $N_{54}(n)$ agrees with every published count in the range $400\le n\le500$ from~\citet{GDK2012}: all $58$ positive entries and all $43$ omitted zero entries agree. The companion scripts verify this comparison and reconstruct the earlier counts through length $64$. The $96$ classes in Table~\ref{tab:primitive-counts} are sequence-equivalence classes; different such classes are not asserted to yield inequivalent weighing matrices under unrestricted matrix equivalence.
\section{Consequences for composite lengths and weighing matrices}\label{sec:consequences}
The existence bases $D_{54},D_{63},D_{72}$ are minimal under divisibility. They are not lists of all primitive lengths. For example, a primitive $(5,4)$ pair can have length $25$, even though existence at that length already follows from a length-five pair. The two statements ``some pair exists by zero insertion'' and ``this particular pair is obtained from that smaller length'' are distinct.

Corollary~\ref{cor:all-length-answers} proves the nonexistence conjecture in the corrected formulation given by \citet[Conjecture 3]{G2010}. It also gives existence and nonexistence conditions for a broader class of composite lengths.
\begin{corollary}\label{cor:all-length-answers}
No weight-nine pair exists at $3^a p$ when $a\ge0$ is an integer and $p$ is a prime outside
\[
\{5,7,11,13,17,19,23,29,41,43\}.
\]
If every prime factor of a positive integer $m$ exceeds $43$, a pair exists at $2^b3^a m$, for integers $a,b\ge0$, exactly when $b\ge3$ or when $b\ge2$ and $a\ge1$. In particular, no pair exists at $3^a m$ or $2\cdot3^a m$.
\end{corollary}
\begin{proof}
Under the first restriction, no member of $D_{54}$ divides $3^a p$. Under the second, only $8$ and $12$ can divide $2^b3^a m$. Apply Theorem~\ref{thm:main-spectrum}.
\end{proof}

For lengths $3^a$ and $2\cdot3^a$, Appendix~\ref{app:three-obstruction} gives a direct algebraic proof of nonexistence that does not use the finite classification. It also proves that nonexistence remains valid when any finite number of additional integer coordinates is allowed.

For an explicit primitive composite example, take
\[
 A=(-1,1,1,1,0,0,0,1),\qquad B=(1,0,0,1,-1,-1,0,0).
\]
Their PAFs are $(5,0,1,0,2,0,1,0)$ and $(4,0,-1,0,-2,0,-1,0)$. The pair has total weight nine and zero PAF. Every proper divisor of eight is at most four, where two ternary sequences have total weight at most eight, so it is primitive.

Theorem~\ref{thm:main-spectrum} also decides when a weighing matrix $W(2n,9)$ can be obtained by the two-circulant array~\eqref{eq:weighing-array}. For example, length $47$ is excluded, whereas length $235=5\cdot47$ is admissible. Nonexistence at a prime length therefore does not extend to all its multiples. A pair at length $235$ is obtained by multiplying every support index of the length-five example in Table~\ref{tab:seeds54} by $47$.
\section{Existence constructions and generation of all pairs}\label{sec:seed-constructions}
\subsection{A seed pair at every minimal admissible length}
Tables~\ref{tab:seeds54}--\ref{tab:seeds90} give $33$ seed pairs, one at every minimal admissible divisor for each distribution. They construct at least one pair at every admissible length; generating every pair requires the full $96$-representative catalogue in Appendices~\ref{app:catalogue} and~\ref{app:scalar-catalogue}, as explained in Section~\ref{sec:all-pair-generation}. All positions are numbered from zero. The positive and negative describing sets $A_+,A_-,B_+,B_-$ specify the nonzero entries; all other entries are zero.
The supports for distributions $(5,4)$, $(6,3)$ and $(9,0)$, and for the prime-length $(7,2)$ examples, are taken from Table~1 of~\citet{G2010}, with harmless component sign changes. The five composite $(7,2)$ representatives were regenerated at parameters already known from Table~4 of~\citet{G2010}. Every displayed pair has been verified directly by integer periodic autocorrelation. The companion data file and verification program reproduce these checks.
\begin{table}[htbp]
\centering\small
\caption{Explicit seed pairs of distribution $(5,4)$.}\label{tab:seeds54}
\begin{tabular}{rllll}
\toprule
$n$&$A_+$&$A_-$&$B_+$&$B_-$\\
\midrule
5&$\{1,2,3,4\}$&$\{0\}$&$\{0,1\}$&$\{2,4\}$\\
7&$\{1,2,3,4\}$&$\{0\}$&$\{0,3\}$&$\{2,6\}$\\
8&$\{1,2,3,7\}$&$\{0\}$&$\{0,3\}$&$\{4,5\}$\\
11&$\{1,2,3,9\}$&$\{0\}$&$\{0,2\}$&$\{3,6\}$\\
12&$\{1,3,5,6\}$&$\{0\}$&$\{0,4\}$&$\{2,8\}$\\
13&$\{1,2,3,5\}$&$\{0\}$&$\{0,5\}$&$\{4,11\}$\\
17&$\{1,2,6,8\}$&$\{0\}$&$\{0,8\}$&$\{10,12\}$\\
19&$\{1,2,3,6\}$&$\{0\}$&$\{0,9\}$&$\{4,10\}$\\
23&$\{1,2,5,7\}$&$\{0\}$&$\{0,7\}$&$\{3,13\}$\\
29&$\{1,2,9,21\}$&$\{0\}$&$\{0,5\}$&$\{10,12\}$\\
41&$\{1,2,4,38\}$&$\{0\}$&$\{0,12\}$&$\{5,29\}$\\
43&$\{1,2,4,40\}$&$\{0\}$&$\{0,12\}$&$\{5,24\}$\\
\bottomrule
\end{tabular}
\end{table}
\begin{table}[htbp]
\centering\small
\caption{Explicit seed pairs of distribution $(6,3)$.}\label{tab:seeds63}
\begin{tabular}{rllll}
\toprule
$n$&$A_+$&$A_-$&$B_+$&$B_-$\\
\midrule
7&$\{3,4,6\}$&$\{0,1,5\}$&$\{0,4,5\}$&$\varnothing$\\
10&$\{4,5,9\}$&$\{0,1,3\}$&$\{0,2,6\}$&$\varnothing$\\
11&$\{2,4,9\}$&$\{0,1,8\}$&$\{0,2,3\}$&$\varnothing$\\
12&$\{4,7,10\}$&$\{0,1,8\}$&$\{0,2,4\}$&$\varnothing$\\
13&$\{3,6,12\}$&$\{0,1,2\}$&$\{0,2,5\}$&$\varnothing$\\
15&$\{1,4,9\}$&$\{0,2,6\}$&$\{0,1,2\}$&$\varnothing$\\
17&$\{4,11,15\}$&$\{0,1,3\}$&$\{0,3,8\}$&$\varnothing$\\
19&$\{2,7,15\}$&$\{0,1,13\}$&$\{0,15,17\}$&$\varnothing$\\
23&$\{5,16,21\}$&$\{0,1,4\}$&$\{0,6,21\}$&$\varnothing$\\
25&$\{1,6,15\}$&$\{0,4,10\}$&$\{0,1,3\}$&$\varnothing$\\
29&$\{21,26,28\}$&$\{0,1,4\}$&$\{0,8,17\}$&$\varnothing$\\
\bottomrule
\end{tabular}
\end{table}
\begin{table}[htbp]
\centering\small
\caption{Explicit seed pairs of distribution $(7,2)$.}\label{tab:seeds72}
\begin{tabular}{rllll}
\toprule
$n$&$A_+$&$A_-$&$B_+$&$B_-$\\
\midrule
10&$\{2,3,5,7,9\}$&$\{0,1\}$&$\{0\}$&$\{4\}$\\
12&$\{1,5,6,7,11\}$&$\{0,3\}$&$\{0\}$&$\{6\}$\\
13&$\{0,2,3,5,9\}$&$\{11,12\}$&$\{0\}$&$\{1\}$\\
14&$\{3,4,6,7,13\}$&$\{0,2\}$&$\{0\}$&$\{2\}$\\
15&$\{4,5,6,10,14\}$&$\{0,3\}$&$\{0\}$&$\{6\}$\\
19&$\{6,12,15,16,17\}$&$\{0,7\}$&$\{0\}$&$\{1\}$\\
23&$\{11,14,18,19,20\}$&$\{0,12\}$&$\{0\}$&$\{1\}$\\
25&$\{1,3,16,17,20\}$&$\{0,5\}$&$\{0\}$&$\{6\}$\\
\bottomrule
\end{tabular}
\end{table}
\begin{table}[htbp]
\centering\small
\caption{Explicit seed pairs of distribution $(9,0)$.}\label{tab:seeds90}
\begin{tabular}{rllll}
\toprule
$n$&$A_+$&$A_-$&$B_+$&$B_-$\\
\midrule
13&$\{2,3,5,7,9,10\}$&$\{0,1,4\}$&$\varnothing$&$\varnothing$\\
24&$\{1,5,12,13,17,18\}$&$\{0,6,9\}$&$\varnothing$&$\varnothing$\\
\bottomrule
\end{tabular}
\end{table}

Only one suitable seed is needed to construct a pair at a prescribed admissible length. Recall that uniform zero insertion with an integer factor $t\ge1$ inserts $t-1$ zeros after every entry of both components, including the last, changing length $m$ to $tm$. The following verification justifies this use of the tables; it does not replace the complete primitive catalogue.

\begin{proposition}\label{prop:verified-seeds}
Every pair in Tables~\ref{tab:seeds54}--\ref{tab:seeds90} has total weight nine and zero PAF. Consequently, these tables, together with uniform zero insertion, establish existence for every admissible length and weight distribution.
\end{proposition}
\begin{proof}
Counting the positions in the positive and negative describing sets gives the stated component weights, which sum to nine. For each displayed pair, subtract the multiset of differences joining opposite signs from the multiset joining equal signs, in each component, and add the two results. Every nonzero residue has coefficient zero, while the zero residue has coefficient nine. This is a finite direct verification of the displayed constructions, reproduced using exact integer arithmetic by the supplied program. Uniform zero insertion preserves the group-ring identity $AA^*+BB^*=9$.
\end{proof}

For an admissible length $n=tm$ and a seed supported on $C_m$, replace each support index $j$ by $tj\pmod n$, preserving its sign. More generally, independent cyclic permutations of the components and a common unit multiplier give
\[
 A_n(x)=x^a A_m(x^{tu}),\qquad B_n(x)=x^b B_m(x^{tu}),\qquad \gcd(u,n)=1.
\]
Every pair constructed in this way has total weight nine and zero PAF, by Proposition~\ref{prop:gcd-contraction}. The supplied seed-construction program selects a valid seed and returns one pair of inflated signed supports for any requested admissible length and distribution. Its output is an existence construction, not an enumeration of every class at that length.

\subsection{Generating every pair at a prescribed length}\label{sec:all-pair-generation}
For a fixed $n$ and one of the distributions $(5,4)$, $(6,3)$, $(7,2)$ or $(9,0)$, retain every representative $(P,Q)$ of that distribution in Appendices~\ref{app:catalogue} and~\ref{app:scalar-catalogue} whose length $m$ divides $n$. Put $t=n/m$. The complete set of pairs of that distribution at length $n$ is given, in $\Z[C_n]$, by
\begin{equation}\label{eq:all-pairs-formula}
\begin{aligned}
 A(x)&=\varepsilon_A x^a P(x^{tu\delta_A}),\\
 B(x)&=\varepsilon_B x^b Q(x^{tu\delta_B}),
\end{aligned}
\qquad
\begin{gathered}
 a,b\in\Z/n\Z,\quad \gcd(u,n)=1,\\
 \varepsilon_A,\varepsilon_B,\delta_A,\delta_B\in\{1,-1\}.
\end{gathered}
\end{equation}
Here $t$ inserts $t-1$ zeros between consecutive entries, $a,b$ specify independent cyclic permutations, $\varepsilon_A,\varepsilon_B$ change component signs, $\delta_A,\delta_B$ reverse components independently, and $u$ is the common unit multiplier. These parameters can yield duplicate pairs, but the primitive representative from which a pair arises is unique under the stated equivalence. Theorem~\ref{thm:complete-generation} proves that every weight-nine pair occurs in~\eqref{eq:all-pairs-formula}. For $(9,0)$ one takes $Q=0$, so the parameters acting on $B$ are immaterial. Interchanging the components covers the opposite ordering of their weights. No search for additional primitive pairs at larger lengths is required.

For example, at length $25$ in distribution $(5,4)$, Appendix~\ref{app:catalogue} supplies one primitive representative of length five and two of length $25$. Inflating the length-five representative with factor five and using the two length-$25$ representatives with factor one gives all three sequence-equivalence classes at length $25$. The single length-five existence seed alone supplies only the first of these classes. Similarly, at length $26$ in distribution $(9,0)$, the two primitive length-$26$ classes cannot be obtained by inflating length-$13$ sequences, although $13\mid26$ already settles existence. This illustrates why the complete primitive catalogue is needed to generate every solution, even when a shorter seed already establishes existence.

\FloatBarrier
\section{Conclusion}
The missing step in the earlier weight-nine investigations was a bound that excludes further primitive pairs at arbitrarily large lengths. The rigidity argument supplies this step for two nonzero components. The signed-difference equations then give bounds of $64$, $48$ and $64$ for the three weight distributions. Exhaustive calculation within these ranges completes the classification. The known constructions therefore give all admissible lengths, and the implication and nonexistence conjectures discussed above hold without a length restriction.

The results can be used at different levels. To decide whether a pair exists at a given length, one tests divisibility by the corresponding base in Theorem~\ref{thm:main-spectrum}. A successful test identifies a seed that can be extended by uniform zero insertion. If every solution is required, Theorem~\ref{thm:complete-generation} supplies the complete set of primitive representatives to use. The counting formula gives the number of classes without constructing the longer sequences. This distinction matters at composite lengths, where additional primitive pairs can coexist with zero insertions of shorter pairs.

Together with the five established circulant classes, the catalogue gives all $96$ primitive sequence-equivalence classes. It supplies every sequence input to the two-circulant construction of weight-nine weighing matrices. It does not decide equivalence under arbitrary matrix operations. The main methodological benefit is a proved stopping point for exhaustive search. At other weights, the corresponding rigidity statement would need a separate proof. Appendix~\ref{app:three-obstruction} gives a direct algebraic proof for the excluded lengths $3^a$ and $2\cdot3^a$, and extends their nonexistence to any finite number of additional integer coordinates.

\paragraph{Reproducibility.}
Online Resource~1 (\texttt{ESM\_1.zip}) contains the search programs, the complete catalogue and results for each length. Its single \texttt{readme.txt} explains every file. The command \texttt{python reproduce.py} checks the stored representatives, seeds and count comparisons. The command \verb|python reproduce.py --full| repeats the exhaustive searches within the proved bounds, including the three known primitive scalar lengths. These searches establish completeness in the bounded ranges; checking the stored representatives alone does not establish nonexistence at unlisted lengths.

\appendix
\section{Nonexistence for cyclic orders \texorpdfstring{$3^a$ and $2\cdot3^a$}{3\textasciicircum a and 2 times 3\textasciicircum a} with additional integer coordinates}\label{app:three-obstruction}
The main existence theorem excludes lengths $3^a$ and $2\cdot3^a$. We give a direct algebraic proof of these exclusions, independent of the finite enumeration, and extend them to support positions with any finite number of additional integer coordinates, each taking values in $\Z$. This explains these infinite families of excluded lengths and rules out the corresponding constructions with additional integer coordinates. The main classification does not depend on this extension. We first extend the reduction modulo three used in Lemma~\ref{lem:augmentation-free}.

\begin{lemma}\label{lem:mod3-laurent-anisotropic}
For Laurent polynomials $F,G$ in finitely many variables over
$\mathbb F_3$, with the operation $*$ inverting all variables, the identity
$FF^*+GG^*=0$ implies $F=G=0$.
\end{lemma}
\begin{proof}
If either polynomial is nonzero, choose integer exponents $u_1,\ldots,u_t$
such that substitution $X_i=x^{u_i}$ is injective on the finite union of
their supports. Their images are not both zero and still satisfy the
identity. At $x=1$, let $r$ be the smaller of their vanishing orders. The
coefficient of $(x-1)^{2r}$ in their norm sum is
$(-1)^r(a^2+b^2)\ne0$ over $\mathbb F_3$, a contradiction.
\end{proof}

The lemma prevents cancellation after reduction modulo three. Evaluation at a root of unity then forces repeated coefficient blocks. These make the evaluated pair vanish, contradicting the complementarity identity.

\begin{theorem}\label{thm:three-times-free}
There is no ternary complementary pair of total weight nine in
$\mathbb Z[C_h\times\mathbb Z^t]$ when $h=3^a$ or $h=2\cdot3^a$, with
integers $a,t\ge0$. Here ternary means that each group element has coefficient
$0$, $1$, or $-1$ in each component, and both components have finite
support.
\end{theorem}
\begin{proof}
Suppose $a\ge1$, put $b=3^{a-1}$, and evaluate the cyclic generator at
$\zeta$, a primitive $3^a$th root. The images of the two components are
Laurent polynomials $\alpha,\beta$ in $t$ free variables over
$\mathcal O=\mathbb Z[\zeta]$. The operation $*$ now conjugates coefficients and inverts the free variables. Set $\pi=1-\zeta$ and $e=2b$. For $c\in\mathcal O$, let $v_\pi(c)$ be the exponent of $\pi$ dividing $c$, with $v_\pi(0)=\infty$. By the standard cyclotomic ramification facts~\citep{Washington1997}, the prime
three is totally ramified with $v_\pi(3)=e$, residue field $\mathbb F_3$,
and $\overline\pi/\pi\equiv-1\pmod\pi$, where the bar denotes complex conjugation. Extend $v_\pi$ to these Laurent
polynomials by taking the minimum valuation of their coefficients.

Let $r=\min(v_\pi(\alpha),v_\pi(\beta))$. Divide by $\pi^r$ and reduce
modulo $\pi$. At least one of the resulting residue Laurent polynomials
$F,G$ is nonzero, and the leading residue of the norm sum is
$(-1)^r(FF^*+GG^*)$. Lemma \ref{lem:mod3-laurent-anisotropic} says that
this does not vanish. Hence
\[
 v_\pi(\alpha\alpha^*+\beta\beta^*)=2r.
\]
The right side of the original identity is nine, with valuation $2e$,
so $r=e$. Every coefficient of $\alpha$ and $\beta$ consequently lies in
$3\mathcal O$.

For each fixed free exponent $z$, let $a_z(y)$ and $b_z(y)$ be the cyclic
coefficient polynomials of the original pair. If $h=3^a=3b$, their degrees
are below $3b$ and their coefficients are ternary. Write $\Phi_m$ for the $m$th cyclotomic polynomial. Since
$\mathcal O/3\mathcal O\cong\mathbb F_3[y]/(\Phi_{3^a}(y)\bmod3)$,
divisibility of
$a_z(\zeta)$ and $b_z(\zeta)$ by three says that their reductions modulo
three are divisible by
\[
 \Phi_{3^a}(y)=1+y^b+y^{2b}.
\]
Thus their three blocks of length $b$ have equal coefficients modulo
three. Since $-1,0,1$ represent three distinct residues, those blocks are
equal over the integers. Every $a_z(\zeta)$ and $b_z(\zeta)$ is therefore
zero, contrary to $\alpha\alpha^*+\beta\beta^*=9$.

If $h=2\cdot3^a=6b$, repeat the same valuation argument at $-\zeta$.
The reductions of every cyclic coefficient polynomial are now divisible
by both $\Phi_{3^a}(y)$ and $\Phi_{2\cdot3^a}(y)$, which are relatively
prime modulo three. Their product is
\[
 (1+y^b+y^{2b})(1-y^b+y^{2b})=1+y^{2b}+y^{4b}.
\]
Each original polynomial has degree below $6b$, so its quotient by this
degree-$4b$ product has degree below $2b$. Its three coefficient blocks
of length $2b$ are therefore equal modulo three and hence equal over the
integers. Evaluation at $\zeta$ again
annihilates every block polynomial and gives a contradiction.

For $a=0$, evaluate the cyclic generator at one. For $h=1$ or
$h=2$, each coefficient in either resulting Laurent polynomial has
absolute value at most two. Generic specialisation of the free variables
to one variable, injective on their supports, and Lemma
\ref{lem:augmentation-free} would require a coefficient of absolute
value three. This is impossible.
\end{proof}

\section{Complete primitive catalogue for the nonzero distributions}\label{app:catalogue}
The tables give all $91$ primitive classes with two nonzero components under Definition~\ref{def:equivalence}. Their indices are zero-based, and $A_+,A_-,B_+,B_-$ are the positive and negative describing sets. Unlisted entries are zero. The records were generated as in Section~\ref{sec:enumeration} and checked for zero PAF and pairwise inequivalence. Use every representative whose length divides the required length in formula~\eqref{eq:all-pairs-formula}.

{\small\setlength{\tabcolsep}{5pt}
\begin{longtable}{rrllll}
\caption{Primitive representatives of distribution $(5,4)$.}\label{tab:full54}\\
\toprule $n$&Class&$A_+$&$A_-$&$B_+$&$B_-$\\\midrule\endfirsthead
\multicolumn{6}{c}{\tablename\ \thetable\ (continued)}\\\toprule $n$&Class&$A_+$&$A_-$&$B_+$&$B_-$\\\midrule\endhead
\midrule\multicolumn{6}{r}{Continued on the next page}\\\endfoot\bottomrule\endlastfoot
5&1&$\{1,2,3,4\}$&$\{0\}$&$\{0,1\}$&$\{2,4\}$\\
7&1&$\{1,2,3,4\}$&$\{0\}$&$\{0,3\}$&$\{1,4\}$\\
8&1&$\{1,2,3,7\}$&$\{0\}$&$\{0,1\}$&$\{2,5\}$\\
8&2&$\{1,3,4,5\}$&$\{0\}$&$\{0,1\}$&$\{2,7\}$\\
8&3&$\{1,3,4,5\}$&$\{0\}$&$\{0,3\}$&$\{1,6\}$\\
10&1&$\{1,2,3,5\}$&$\{0\}$&$\{0,3\}$&$\{1,6\}$\\
10&2&$\{1,2,3,9\}$&$\{0\}$&$\{0,2\}$&$\{4,8\}$\\
10&3&$\{1,2,7,9\}$&$\{0\}$&$\{0,1\}$&$\{3,6\}$\\
11&1&$\{1,2,3,9\}$&$\{0\}$&$\{0,2\}$&$\{3,6\}$\\
11&2&$\{1,2,3,10\}$&$\{0\}$&$\{0,1\}$&$\{2,7\}$\\
11&3&$\{1,2,4,5\}$&$\{0\}$&$\{0,4\}$&$\{1,7\}$\\
11&4&$\{1,2,5,7\}$&$\{0\}$&$\{0,2\}$&$\{4,8\}$\\
11&5&$\{1,2,5,10\}$&$\{0\}$&$\{0,1\}$&$\{4,8\}$\\
12&1&$\{1,3,5,6\}$&$\{0\}$&$\{0,4\}$&$\{2,8\}$\\
12&2&$\{2,4,8,10\}$&$\{0\}$&$\{0,3\}$&$\{6,9\}$\\
13&1&$\{1,2,3,5\}$&$\{0\}$&$\{0,5\}$&$\{1,7\}$\\
13&2&$\{1,2,4,10\}$&$\{0\}$&$\{0,2\}$&$\{4,8\}$\\
13&3&$\{1,2,6,12\}$&$\{0\}$&$\{0,1\}$&$\{3,5\}$\\
13&4&$\{1,2,8,10\}$&$\{0\}$&$\{0,3\}$&$\{6,9\}$\\
13&5&$\{1,3,10,11\}$&$\{0\}$&$\{0,2\}$&$\{4,7\}$\\
14&1&$\{1,2,5,11\}$&$\{0\}$&$\{0,2\}$&$\{4,6\}$\\
14&2&$\{1,2,7,9\}$&$\{0\}$&$\{0,2\}$&$\{4,8\}$\\
14&3&$\{1,3,7,8\}$&$\{0\}$&$\{0,1\}$&$\{2,5\}$\\
14&4&$\{1,3,7,10\}$&$\{0\}$&$\{0,1\}$&$\{3,6\}$\\
14&5&$\{1,4,7,11\}$&$\{0\}$&$\{0,1\}$&$\{4,9\}$\\
14&6&$\{1,6,7,13\}$&$\{0\}$&$\{0,1\}$&$\{3,6\}$\\
15&1&$\{1,2,3,6\}$&$\{0\}$&$\{0,5\}$&$\{1,10\}$\\
15&2&$\{1,2,4,8\}$&$\{0\}$&$\{0,3\}$&$\{6,12\}$\\
15&3&$\{1,2,10,14\}$&$\{0\}$&$\{0,1\}$&$\{4,9\}$\\
15&4&$\{3,5,6,10\}$&$\{0\}$&$\{0,5\}$&$\{1,7\}$\\
16&1&$\{1,3,4,7\}$&$\{0\}$&$\{0,5\}$&$\{2,11\}$\\
16&2&$\{1,3,4,13\}$&$\{0\}$&$\{0,3\}$&$\{5,10\}$\\
17&1&$\{1,2,6,8\}$&$\{0\}$&$\{0,2\}$&$\{4,12\}$\\
17&2&$\{1,2,8,10\}$&$\{0\}$&$\{0,2\}$&$\{4,8\}$\\
17&3&$\{1,4,5,14\}$&$\{0\}$&$\{0,2\}$&$\{4,9\}$\\
19&1&$\{1,2,3,6\}$&$\{0\}$&$\{0,6\}$&$\{1,10\}$\\
19&2&$\{1,2,9,11\}$&$\{0\}$&$\{0,3\}$&$\{6,12\}$\\
19&3&$\{1,2,9,18\}$&$\{0\}$&$\{0,1\}$&$\{4,8\}$\\
19&4&$\{1,2,12,14\}$&$\{0\}$&$\{0,4\}$&$\{8,13\}$\\
20&1&$\{1,2,3,19\}$&$\{0\}$&$\{0,6\}$&$\{2,14\}$\\
20&2&$\{1,2,7,9\}$&$\{0\}$&$\{0,3\}$&$\{6,15\}$\\
20&3&$\{1,2,13,15\}$&$\{0\}$&$\{0,3\}$&$\{6,11\}$\\
20&4&$\{2,4,6,12\}$&$\{0\}$&$\{0,1\}$&$\{2,11\}$\\
20&5&$\{2,6,8,16\}$&$\{0\}$&$\{0,3\}$&$\{6,13\}$\\
21&1&$\{1,7,8,14\}$&$\{0\}$&$\{0,3\}$&$\{6,15\}$\\
21&2&$\{1,8,9,18\}$&$\{0\}$&$\{0,3\}$&$\{7,14\}$\\
22&1&$\{1,2,3,21\}$&$\{0\}$&$\{0,6\}$&$\{2,12\}$\\
22&2&$\{1,2,5,7\}$&$\{0\}$&$\{0,7\}$&$\{3,16\}$\\
22&3&$\{1,2,17,19\}$&$\{0\}$&$\{0,3\}$&$\{7,16\}$\\
23&1&$\{1,2,5,7\}$&$\{0\}$&$\{0,7\}$&$\{3,13\}$\\
23&2&$\{1,2,11,13\}$&$\{0\}$&$\{0,3\}$&$\{6,12\}$\\
25&1&$\{1,2,16,18\}$&$\{0\}$&$\{0,4\}$&$\{8,15\}$\\
25&2&$\{1,5,14,15\}$&$\{0\}$&$\{0,5\}$&$\{9,17\}$\\
28&1&$\{1,7,14,15\}$&$\{0\}$&$\{0,2\}$&$\{4,8\}$\\
28&2&$\{2,4,8,16\}$&$\{0\}$&$\{0,3\}$&$\{6,17\}$\\
28&3&$\{2,4,8,22\}$&$\{0\}$&$\{0,5\}$&$\{10,19\}$\\
29&1&$\{1,2,9,21\}$&$\{0\}$&$\{0,2\}$&$\{7,12\}$\\
41&1&$\{1,2,4,38\}$&$\{0\}$&$\{0,12\}$&$\{5,29\}$\\
43&1&$\{1,2,4,40\}$&$\{0\}$&$\{0,12\}$&$\{5,24\}$\\
\end{longtable}}

{\small\setlength{\tabcolsep}{5pt}
\begin{longtable}{rrllll}
\caption{Primitive representatives of distribution $(6,3)$.}\label{tab:full63}\\
\toprule $n$&Class&$A_+$&$A_-$&$B_+$&$B_-$\\\midrule\endfirsthead
\multicolumn{6}{c}{\tablename\ \thetable\ (continued)}\\\toprule $n$&Class&$A_+$&$A_-$&$B_+$&$B_-$\\\midrule\endhead
\midrule\multicolumn{6}{r}{Continued on the next page}\\\endfoot\bottomrule\endlastfoot
7&1&$\{0,1,3\}$&$\{2,4,5\}$&$\{0,1,3\}$&$\varnothing$\\
10&1&$\{0,1,3\}$&$\{4,5,9\}$&$\{0,2,6\}$&$\varnothing$\\
11&1&$\{0,1,3\}$&$\{2,7,9\}$&$\{0,1,5\}$&$\varnothing$\\
12&1&$\{0,1,5\}$&$\{3,6,9\}$&$\{0,2,4\}$&$\varnothing$\\
13&1&$\{0,1,2\}$&$\{3,6,12\}$&$\{0,2,5\}$&$\varnothing$\\
13&2&$\{0,1,4\}$&$\{2,7,9\}$&$\{0,2,7\}$&$\varnothing$\\
13&3&$\{0,1,4\}$&$\{2,8,10\}$&$\{0,2,6\}$&$\varnothing$\\
15&1&$\{0,1,3\}$&$\{2,8,12\}$&$\{0,1,8\}$&$\varnothing$\\
15&2&$\{0,1,3\}$&$\{2,11,12\}$&$\{0,4,8\}$&$\varnothing$\\
15&3&$\{0,1,5\}$&$\{2,10,12\}$&$\{0,3,6\}$&$\varnothing$\\
17&1&$\{0,1,3\}$&$\{4,11,15\}$&$\{0,3,8\}$&$\varnothing$\\
19&1&$\{0,1,3\}$&$\{2,6,7\}$&$\{0,6,12\}$&$\varnothing$\\
19&2&$\{0,1,4\}$&$\{3,8,14\}$&$\{0,2,9\}$&$\varnothing$\\
19&3&$\{0,1,8\}$&$\{2,11,15\}$&$\{0,2,5\}$&$\varnothing$\\
20&1&$\{0,1,9\}$&$\{3,10,13\}$&$\{0,2,6\}$&$\varnothing$\\
21&1&$\{0,1,5\}$&$\{10,18,20\}$&$\{0,3,9\}$&$\varnothing$\\
21&2&$\{0,1,8\}$&$\{9,15,18\}$&$\{0,4,11\}$&$\varnothing$\\
23&1&$\{0,1,4\}$&$\{5,16,21\}$&$\{0,2,8\}$&$\varnothing$\\
25&1&$\{0,1,5\}$&$\{6,16,22\}$&$\{0,3,11\}$&$\varnothing$\\
29&1&$\{0,1,4\}$&$\{21,26,28\}$&$\{0,8,17\}$&$\varnothing$\\
\end{longtable}}

{\small\setlength{\tabcolsep}{5pt}
\begin{longtable}{rrllll}
\caption{Primitive representatives of distribution $(7,2)$.}\label{tab:full72}\\
\toprule $n$&Class&$A_+$&$A_-$&$B_+$&$B_-$\\\midrule\endfirsthead
\multicolumn{6}{c}{\tablename\ \thetable\ (continued)}\\\toprule $n$&Class&$A_+$&$A_-$&$B_+$&$B_-$\\\midrule\endhead
\midrule\multicolumn{6}{r}{Continued on the next page}\\\endfoot\bottomrule\endlastfoot
10&1&$\{0,1,2,4,6\}$&$\{5,8\}$&$\{0\}$&$\{2\}$\\
12&1&$\{0,1,2,4,9\}$&$\{5,11\}$&$\{0\}$&$\{6\}$\\
12&2&$\{0,1,2,6,8\}$&$\{4,7\}$&$\{0\}$&$\{6\}$\\
12&3&$\{0,1,3,5,6\}$&$\{2,10\}$&$\{0\}$&$\{6\}$\\
13&1&$\{0,1,2,6,9\}$&$\{5,7\}$&$\{0\}$&$\{2\}$\\
14&1&$\{0,1,3,4,8\}$&$\{5,7\}$&$\{0\}$&$\{2\}$\\
15&1&$\{0,1,2,6,10\}$&$\{11,14\}$&$\{0\}$&$\{6\}$\\
19&1&$\{0,1,2,4,13\}$&$\{3,8\}$&$\{0\}$&$\{2\}$\\
23&1&$\{0,1,2,6,9\}$&$\{8,20\}$&$\{0\}$&$\{1\}$\\
25&1&$\{0,1,3,7,17\}$&$\{18,23\}$&$\{0\}$&$\{4\}$\\
28&1&$\{0,1,3,8,19\}$&$\{20,26\}$&$\{0\}$&$\{6\}$\\
28&2&$\{0,1,3,10,23\}$&$\{4,8\}$&$\{0\}$&$\{10\}$\\
\end{longtable}}

\section{Complete primitive catalogue for the scalar distribution}\label{app:scalar-catalogue}
These are the five established primitive $(9,0)$ classes, with $B=0$. Proposition~\ref{prop:scalar-five} verifies their completeness from the cited circulant results. All indices are zero-based and unlisted entries are zero. Together with Appendix~\ref{app:catalogue}, they complete the $96$ representatives used in formula~\eqref{eq:all-pairs-formula}.

\begin{table}[htbp]\centering\small
\caption{All five primitive representatives of distribution $(9,0)$; $B=0$.}\label{tab:full90}
\begin{tabular}{rrll}\toprule
$n$&Class&$A_+$&$A_-$\\\midrule
13&1&$\{2,5,6,7,8,11\}$&$\{1,3,9\}$\\
13&2&$\{4,7,8,10,11,12\}$&$\{1,3,9\}$\\
24&1&$\{3,4,8,9,16,20\}$&$\{0,15,21\}$\\
26&1&$\{1,3,4,9,10,12\}$&$\{2,6,18\}$\\
26&2&$\{7,11,14,16,21,22\}$&$\{2,6,18\}$\\
\bottomrule\end{tabular}\end{table}

\setlength{\bibsep}{3pt}

\end{document}